\documentclass{article}[11pt]
\pdfoutput=1
\usepackage[dvips]{color}

\usepackage[
pdftitle={An integral equation formulation for rigid bodies in Stokes flow in three dimensions},
pdfcreator={pdflatex},
pdfsubject={preprint},
hyperindex = {true},
colorlinks = {true},
linkcolor = {blue},
pagecolor = {blue},
citecolor = {blue}
]{hyperref}

\usepackage[square, sort, comma, numbers]{natbib}
\usepackage{tabularx}
\usepackage{pgf}
\usepackage{amsmath,amsbsy}

\usepackage{amsmath,amssymb}
\usepackage{amsmath}
\usepackage{verbatim}
\usepackage{graphicx}
\usepackage{tabularx,booktabs,soul,multirow,collcell,dcolumn}
\usepackage{algorithm}
\usepackage{algorithmic}
\usepackage[export]{adjustbox} 
\usepackage[scriptsize,tight]{subfigure}
\usepackage{xifthen}
\usepackage{caption}

\newlength\commLen
\newcommand{\mcaption}[3]{
  \ifthenelse{\isempty{#2}}
             {\caption{\textit{#3}\label{#1}}}
             {\caption[#2]{\textit{{\sc #2.}~#3}\label{#1}}}
             }

\newcommand{\algcaption}[3]{
        \ifthenelse{\isempty{#3}}
                   {\caption[#1]{{\sc #2.} \label{#1}}}
                   {\caption[#1]{{\sc #2.} \newline\small{#3} \label{#1}}}
        }

\usepackage{amsmath,amscd}
\usepackage{amscd,mathrsfs} 
\usepackage[all]{xy}
\usepackage{float}

\newcommand{\bbm}{\begin{bmatrix}}
\newcommand{\ebm}{\end{bmatrix}}
\newcommand{\beq}{\begin{equation}}
\newcommand{\beqn}{\begin{equation*}}
\newcommand{\eeq}{\end{equation}}
\newcommand{\eeqn}{\end{equation*}}

\usepackage{prettyref}
\def\figkeyword{Fig.~}
\newcommand\reftext[3][]{\hyperref[#2]{#3\ref{#2}#1}}
\newcommand\eqreftext[3][]{\hyperref[#2]{#3\eqref{#2}#1}}
\newrefformat{sec}{\reftext{#1}{\S}}
\newrefformat{ssc}{\reftext{#1}{\S}}
\newrefformat{tbl}{\reftext{#1}{Table~}}
\newrefformat{fig}{\reftext{#1}{\figkeyword}}
\newrefformat{chp}{\reftext{#1}{Chapter~}}
\newrefformat{eq}{\eqreftext{#1}{Eq.~}}
\newrefformat{set}{\eqreftext{#1}{Eq.~Set~}}
\newrefformat{alg}{\reftext{#1}{Algorithm~}}
\newrefformat{apx}{\reftext{#1}{}}
\newrefformat{lin}{\reftext{#1}{line~}}
\newrefformat{rem}{\reftext{#1}{Remark~}}
\newrefformat{thm}{\reftext{#1}{Theorem~}}
\newrefformat{cor}{\reftext{#1}{Corollary~}}
\newcommand\pr[1]{\prettyref{#1}}

\numberwithin{equation}{section}

\newcommand{\cA[1]}{\mathcal{#1}}

\newcommand{\vc[1]}{\boldsymbol{#1}}

\def\bigO{\mathscr{O}}
\renewcommand\O[1]{\ensuremath{\bigO\left(#1\right)}}

\def\Xint#1{\mathchoice
   {\XXint\displaystyle\textstyle{#1}}%
   {\XXint\textstyle\scriptstyle{#1}}%
   {\XXint\scriptstyle\scriptscriptstyle{#1}}%
   {\XXint\scriptscriptstyle\scriptscriptstyle{#1}}%
   \!\int}
\def\XXint#1#2#3{{\setbox0=\hbox{$#1{#2#3}{\int}$}
     \vcenter{\hbox{$#2#3$}}\kern-.5\wd0}}

\def\dashint{\Xint-}

\usepackage{soul,marginnote,todonotes,cprotect}

\usepackage{amsthm} 
\newtheorem{thm}{Theorem}
\newtheorem{lem}{Lemma}

\newcommand{\bR}{\mathbb{R}}

\title{An integral equation formulation for rigid bodies in Stokes flow in three dimensions}
\author{Eduardo Corona~\footnotemark[1] , Leslie Greengard~\footnotemark[2] , Manas Rachh~\footnotemark[2] , 
Shravan Veerapaneni~\footnotemark[1]}
\date{}
\begin{document}
\maketitle

\renewcommand{\thefootnote}{\fnsymbol{footnote}}
\footnotetext[1]{Department of Mathematics, University of Michigan}
\footnotetext[2]{Courant Institute of Mathematical Sciences, New York University}

\begin{abstract}
We present a new derivation of a boundary integral equation (BIE) for 
simulating the three-dimensional dynamics of arbitrarily-shaped rigid particles of genus zero immersed in a Stokes fluid, on which are prescribed
forces and torques. Our method is based on a single-layer representation
and leads to a simple second-kind integral equation. It avoids
the use of auxiliary sources within each particle that play a role in
some classical formulations. We use a spectrally accurate quadrature scheme to evaluate the corresponding layer potentials, so that only a small number of 
spatial discretization points per particle are required. The resulting discrete sums are computed in $\mathcal{O}(n)$ time, where $n$ denotes the number of particles, using the fast multipole method (FMM). The particle positions and orientations are updated by a high-order time-stepping scheme. We illustrate
the accuracy, conditioning and scaling of our solvers with several
 numerical examples.
\end{abstract}

\section{Introduction}
\label{sec:intro}
In viscous flows, the mobility 
problem consists of computing the translational and 
rotational velocities $(\vc[v]_i,\vc[\omega]_i)$ induced
on a collection of $n$ rigid bodies when prescribed 
forces and torques $(\vc[F]_i,\vc[T]_i)$ are specified on each one.
The Stokes equations, which are linear, govern the ambient viscous fluid at vanishing Reynolds number limit. Thereby, there exists a well-defined 
\emph{mobility matrix} denoted by
$M$ such that
\beq \vc[V] = M \vc[F] \, , \label{eq:mobility}\eeq
where $\vc[V] = (\vc[v]_1,\vc[\omega]_1,\dots,\vc[v]_n,\vc[\omega]_n)$ and $\vc[F] = (\vc[F]_1,\vc[T]_1,\dots,\vc[F]_n,\vc[T]_n)$.

Reformulating the problem as an integral equation has several
advantages over direct discretization of the governing partial
differential equations themselves. First,
integral equation methods require discretization of the particle 
boundaries alone, which leads to an immediate reduction in the size of the
discretized linear system.  Equally important,
carefully chosen integral representations result in 
well-conditioned linear systems, while 
discretizing the Stokes equations directly
leads to highly ill-conditioned systems.
Moreover, the integral representation can be chosen
to satisfy the far field boundary conditions
necessary to model an open system, thereby eliminating the need
for artificial truncation of  the computational domain. 
Lastly, combining high-order quadrature methods and suitable
fast algorithms, boundary integral equations for complex geometries
can be solved to high accuracy in optimal or near optimal time 
\cite{rokhlin1997,CHEW,LIU,Nishimura,kifmm04ying}.

A variety of integral 
representations for the mobility problem have been introduced, 
often in the form of 
first kind integral equations
\cite{Cortez2005} or second kind integral equations with $n$ additional
unknowns and an equal number of additional constraints
\cite{Kropinski1999}. While these have been shown to be very effective,
it is advantageous
(when $n$ is large and the rigid bodies have complicated shape)
to work with well-conditioned second kind boundary integral formulations which
are free of additional constraints.
Such schemes have been developed earlier \cite{karrila1989integral}, 
using the Lorentz reciprocal identity. The equation which we derive below
in Section \ref{sec:formulation} is essentially the same, but obtained
using a different principle - namely that 
the interior of a rigid body must be stress free. 
The mobility problem can also be solved 
using a double layer representation that doesn't involve  additional unknowns.
For a detailed discussion of the latter approach,
we refer the reader to \cite{pozrikidis1992boundary,af2014fast}. 
Our formulation has the advantage that certain derivative quantities, 
such as fluid stresses, can be computed using integral operators with
 weakly singular kernels instead of hypersingular ones, 
 simplifying the quadrature issues. 

Based on this formulation, we present a numerical algorithm in Section \ref{sec:discretization} to solve the mobility problem and evolve the position and orientation of the rigid bodies. 
In Section \ref{sec:numerical}, we discuss several applications and present results from our numerical experiments (a sample simulation with large $n$ is depicted in Figure \ref{eyecandy}). 

\begin{figure}[H]
\centering
\includegraphics[width=\textwidth]{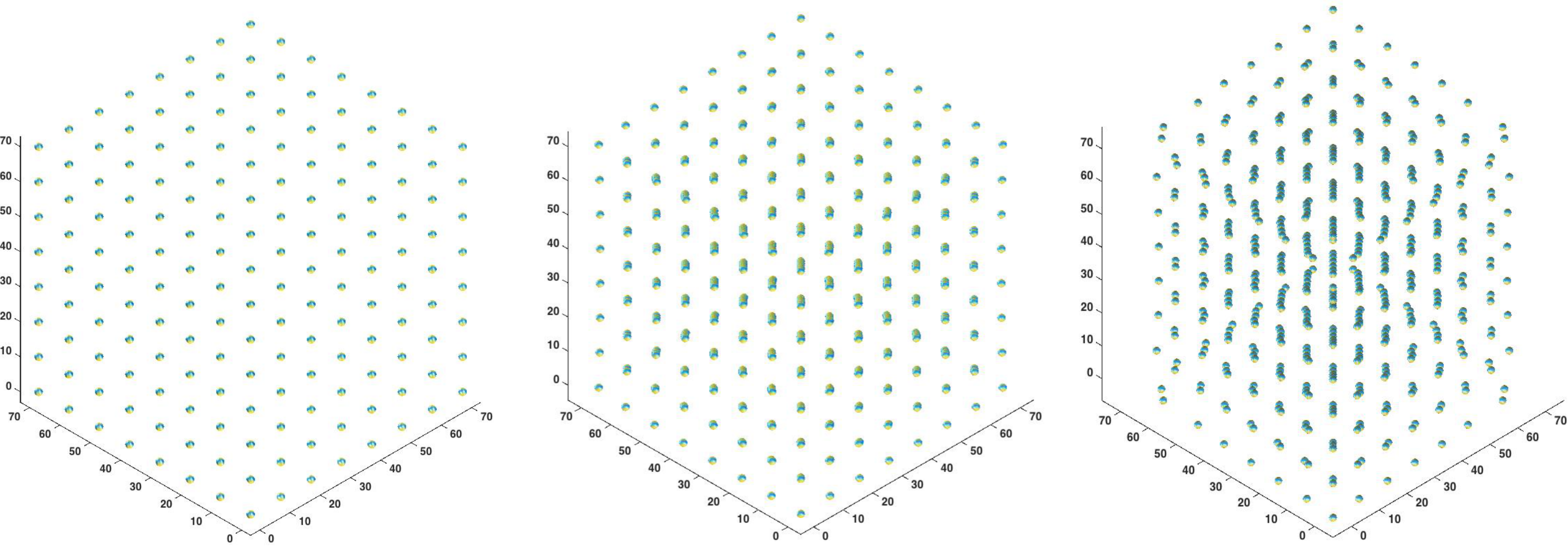}
\mcaption{fig:MHD_8x8x8_snapshots}{Self-assembly of chains in a magnetorheological fluid} {Snapshots from a simulation of a cubic lattice of $512$ paramagnetic spheres subjected to a uniform magnetic field $\vc[H]_0$ in a Stokesian flow. We use spherical harmonic expansions of degree $p=8$ to represent functions on each sphere, requiring a total of $294912$ degrees-of-freedom per time-step to compute hydrodynamic and magnetic interactions. 
At each time-step, a magnetostatic problem is solved for the current particle configuration; the Maxwell stresses thus obtained at the particle boundaries give rise to a mobility poblem (details in Section \ref{sec:numerical}). Our BIE formulations for both the fluid velocity and magnetic potential problems lead to well-conditioned linear systems. Summation of far-field interactions is accelerated via Stokes and Laplace FMMs \cite{toolboxfmm}. On an average, this simulation took $5$ minutes per time-step on a single node with Intel~Xeon~E--$2690$~v$2$($3.0$~GHz) processor and $24$ GB of RAM. } \label{eyecandy}
\end{figure}

\section{The mobility problem}
\label{sec:formulation}
Let $\{D_i\}_{i=1}^n$ be a set of $n$ disjoint rigid bodies in $\mathbb{R}^3$ with boundaries $\Gamma_i$. 
Let $\vc[F]_{i},\vc[T]_{i}$ denote the force and torque exerted on $D_{i}$,
and let $\vc[v]_{i},\vc[\omega]_{i}$ denote the translational and rotational
velocity of $D_{i}$.
Let $E$ be the domain exterior to all of the rigid bodies $\{ D_i \}$, and assume 
that the fluid in $E$ is governed by Stokes flow with viscosity $\mu=1$. 
For a given velocity field
$\vc[u](\vc[x]) \in \mathbb{R}^3$ at a point $\vc[x] \in E$, we denote the 
corresponding fluid pressure, strain and stress tensors by
$p$, $e\left(\vc[u] \right)$ and $\sigma$, respectively.
On the surface of the rigid bodies, $\vc[f] = \sigma \cdot \vc[n]$ 
denotes the surface force or surface traction exerted by the fluid 
on the rigid body $D_{i}$. 
For the sake of simplicity, we assume there are no volume forces. 
The governing equations for the mobility problem are then given by: 

\beq -\Delta \vc[u] + \vc[\nabla] p = \vc[0], \ \ \ \ \vc[\nabla] \cdot \vc[u] = 0 \ \ \forall \vc[x] \in E \label{eq:Stokes}, \eeq
\beq \vc[u](\vc[x]) = \vc[v]_i + \vc[\omega]_i \times (\vc[x] - \vc[x]^c_i) \ \ \forall \vc[x] \in \label{eq:RigidBody} \Gamma_i, \eeq
\beq \int_{\Gamma_i} \vc[f] \, dS_{y} = \int_{\Gamma_i} \sigma \cdot \vc[n] \, dS_{y} = -\vc[F]_i, 
\ \ \ \int_{\Gamma_i} (\vc[x] - \vc[x]^c_i) \times \vc[f] \, dS_{y} = -\vc[T]_i, \label{eq:ForceTorque} \eeq
\beq \vc[u](\vc[x]) \rightarrow \vc[0] \ \ \mathrm{as} \ \ |\vc[x]| \rightarrow \infty. \label{eq:Infinity} \eeq
It should be noted that the forces and torques, $\vc[F]_{i}$, and $\vc[T]_{i}$ are {\em known} and the translational and rotational velocities $\vc[v]_{i}$, 
and $\vc[\omega]_{i}$ are {\em unknown}.

Before turning to the integral equation, however, we state and prove
a simple uniqueness result.
To prove uniqueness, we need the following lemma contained in \cite{pozrikidis1992boundary}
\begin{lem} \label{lem:growth3dStokes}
Let $\vc[u]$ solve Stokes equation in the exterior domain $E \subseteq \mathbb{R}^{3}$ defined above,
satisfying the condition \eqref{eq:Infinity}.
Let $B_{R}\left(0\right)$ be the ball of radius $R$ centered at
the origin and let $\partial B_{R}\left(0\right)$ be its boundary. Then,
there exist $M,R_{0}$ such that $\sup_{\partial B_{R}\left(0\right)}\left| \vc[u]\right|\leq\frac{M}{R}$
and $\sup_{\partial B_{R}\left(0\right)}\left| \sigma \right|\leq\frac{M}{R^{2}}$
for all $R\geq R_{0}$. From these decay conditions at $\infty$ it also follows that
\begin{equation}
\lim_{R\to\infty}\int_{\partial B_{R}\left(0\right)} \left(\vc[u],\vc[f] \right) \, dS_{y}  \to0 \, . \label{eq:ForceDecayatInfty_chap23D}
\end{equation}
\end{lem}

\begin{lem}
\label{lem:BoundaryTermZero3D}
If $\vc[u]$ satisfies equations \eqref{eq:RigidBody},
and \eqref{eq:ForceTorque} 
with $\vc[F]_{i} = \mathbf{0}$
and $\vc[T]_{i} = \mathbf{0}$ then
\begin{equation}
\int_{\Gamma_i} \left(\vc[u], \vc[f] \right) \, dS_{y} = 0 \, . 
\end{equation}
\end{lem}
\begin{proof}
\begin{equation}
\int_{\Gamma_i} \left(\vc[u], \vc[f] \right) \, dS_{y} = 
\int_{\Gamma_{i}} \left( \vc[v]_{i}+\boldsymbol{\omega}_{i} \times \left(\vc[x] - 
\vc[x]^{c}_i \right),\vc[f]\right) \, dS_{y} = -\left(\vc[u]_i, \vc[F]_{i} \right) - \left(\vc[\omega]_{i},\vc[T]_{i}\right) = 0 \, .
\end{equation}
\end{proof}

\begin{lem}
\label{Uniqueness mobility3D} If $\vc[u](\vc[x])$
satisifies equations \eqref{eq:Stokes},
\eqref{eq:RigidBody}, \eqref{eq:ForceTorque} and \eqref{eq:Infinity} with $\vc[F]=\mathbf{0}$,
then $\vc[u](\vc[x]) \equiv \mathbf{0}$ in $E$.
\end{lem}
\begin{proof}
Let $\left\langle \cdot,\cdot\right\rangle :\bR^{3\times3}\times\bR^{3\times3}$
be the Frobenius inner product. For large enough $R$,
\begin{align*}
\int_{E\cap B_{R}\left(0\right)}\left\langle e\left(\vc[u]\right),e\left(\vc[u]\right)\right\rangle dV & = 
\int_{E\cap B_{R}\left(0\right)}\left\langle D\vc[u],e\left(\vc[u]\right)\right\rangle dV\\
 & =\int_{\partial(E\cap B_{R}\left(0\right))}\left(\vc[u],e\left(\vc[u]\right)\cdot\mathbf{n}\right) \, dS_{y} 
 -\frac{1}{2}\int_{E\cap B_{R}\left(0\right)}\left(\vc[u],\Delta\vc[u]\right)dV\\
 & =\int_{\partial(E\cap B_{R}\left(0\right))}\left(\vc[u],e\left(\vc[u]\right)\cdot\mathbf{n}\right) \, dS_{y} 
 -\frac{1}{2}\int_{E\cap B_{R}\left(0\right)}\left(\vc[u],\nabla p\right) dV\\
 & =\frac{1}{2}\int_{\partial\left(E\cap B_{R}\left(0\right)\right)}\left(\vc[u],\left(-p\left[\begin{array}{cc}
1 & 0\\
0 & 1
\end{array}\right]+2e\left(\vc[u]\right)\right).\mathbf{n}\right)\, dS_{y} \quad \\
 & =-\frac{1}{2}\sum_{i=1}^{N}\int_{\Gamma_{i}}\left(\vc[u],\vc[f]\right)\, dS_{y}+\frac{1}{2}\int_{\partial B_{R}\left(0\right)}\left(\vc[u],\vc[f]\right) \, dS_{y}\\
 & =\frac{1}{2}\int_{\partial B_{R}\left(0\right)}\left(\vc[u],\vc[f]\right) \, .
\end{align*}
using \eqref{eq:Stokes} and Lemma \ref{lem:BoundaryTermZero3D}.
Taking the limit as $R\to\infty$ in the above expression and using
equation \eqref{eq:ForceDecayatInfty_chap23D},
we get 
\begin{equation}
e\left(\vc[u]\right)\equiv\left[\begin{array}{cc}
0 & 0\\
0 & 0
\end{array}\right] \quad \vc[x] \in E \, .
\end{equation}
Thus, $\vc[u]$ is a rigid body motion. 
Since $\vc[u](\vc[x]) \to\mathbf{0}$
as $\left|\vc[x]\right| \to \infty$, we conclude that $\vc[u]\equiv \mathbf{0}$.
\end{proof}

Our integral equation is based on representing the solution $\vc[u]$ as the sum of two fields: an \emph{incident} field $\vc[u]_{inc}$ that accounts for the net force and torque conditions in \pr{eq:ForceTorque}, and a \emph{scattered} field $\vc[u]_{sc}$ with zero net forces and torques that enforces that $\vc[u]$ is a rigid body motion at each boundary (\pr{eq:RigidBody}). As noted in \cite{rachh15,rachh20152d}, 
if the incident field is defined by uniformly distributed forces and torques, then 
determination of the scattered
field can be interpreted as redistributing those uniformly placed surface forces so as
to eliminate any interior stress.

\subsection{Notation and preliminaries}

Let $\Gamma$ be a smooth surface in $\mathbb{R}^3$, $D^{\mp}$ denote the domains inside and outside $\Gamma$, and $\vc[n]\left(\vc[x]\right)$ the unit outward normal vector at $\vc[x] \in \Gamma$. 
Let $G_{i,j}(\vc[x],\vc[y])$ be the Stokeslet, that is, the fundamental solution to the Stokes equations in free space in $\mathbb{R}^3$, given by: 
\beq G_{i,j}(\vc[x],\vc[y]) = \frac{1}{8\pi} \left( \frac{\delta_{i,j} }{|\vc[x]-\vc[y]|} + \frac{(x_i-y_i)(x_j-y_j)}{|\vc[x]-\vc[y]|^3} \right)  
\, .\eeq
Then the single layer potential operator for the Stokes equation $\cA[S]_{\Gamma}$ is given by 
\beq \cA[S]_{\Gamma}[\vc[\mu]](\vc[x])_i = \int_{\Gamma} G_{i,j}(\vc[x],\vc[y])\mu_{j}(\vc[y]) \, dS_y
\, . \eeq

If we let $\vc[f]$ denote the surface traction corresponding to $\vc[u] = \cA[S]_{\Gamma}[\vc[\mu]]$, then from standard jump relations for the single layer potential, we have: 
\beq f^{\pm}_i(\vc[x]) = \mp \frac{1}{2} \mu_i (\vc[x]) + n_k (\vc[x]) \dashint_{\Gamma} T_{i,j,k}(\vc[x],\vc[y]) \mu_j(\vc[y]) \, dS_y \label{eq:jumptraction} \, , \eeq 
where $\dashint$ is the principal value integral, and $T_{i,j,k}$ is the traction kernel given by: 
\beq T_{i,j,k}(\vc[x],\vc[y]) = -\frac{3}{4\pi} \frac{(x_i-y_i)(x_j-y_j)(x_k-y_k)}{|\vc[x]-\vc[y]|^5}  \eeq

The interior forces and torques are zero, while the exterior ones are equal to the corresponding moments for $\vc[\mu]$: 
\beq \int_{\Gamma_i} \vc[f]^{-}(\vc[y])\, dS_y = 0 \ \ \ \int_{\Gamma_i} (\vc[y] - \vc[x]^c_i) \times \vc[f]^{-}(\vc[y]) \, dS_{y} = 0 \label{eq:Smuint} \eeq
\beq \int_{\Gamma_i} \vc[f]^{+}(\vc[y])\, dS_y = -\int_{\Gamma} \vc[\mu](\vc[y]) \, dS_y \ \ \ \int_{\Gamma_i} (\vc[y] - \vc[x]^c_i) \times \vc[f]^{+}(\vc[y]) \, dS_y 
= -\int_{\Gamma_i} (\vc[y] - \vc[x]^c_i) \times \vc[\mu] (\vc[y]) \, dS_y \label{eq:Smuext} \, . \eeq

Finally, \beq \left|\cA[S]_{\Gamma}[\vc[\mu]] (\vc[x]) \right| \to 0 \ \ \ \text{as} \ \left|\vc[x] \right| \to \infty \, . \label{eq:Sgrowthinf} \eeq

\subsection{The ``incident" and ``scattered" fields}

We first define an incident field $\vc[u]_{inc}$ which is a solution to the Stokes equations in $E$ (Equations (\ref{eq:Stokes}),(\ref{eq:Infinity})) and satisfies the specified net forces and torques on each boundary $\Gamma_i$ (\pr{eq:ForceTorque}). 
Given a force density $\vc[\rho](\vc[x])$ in $\Gamma = \cup_{i=1}^{n} \Gamma_i$, with $\vc[\rho]_i = \vc[\rho] |_{\Gamma_i}$, we let $\vc[u]_{inc} = \cA[S]_{\Gamma}[\vc[\rho]]$. Since the Stokes equations and conditions at infinity are immediately satisfied, according to \pr{eq:Smuext}, we only need to choose a $\vc[\rho]$ such that
\beq \vc[F]_i = \int_{\Gamma_i} \vc[\rho]_{i}(\vc[y]) \, \ dS_y \ \, \ \vc[T]_i = \int_{\Gamma_i} (\vc[y] - \vc[x]^c_i) \times \vc[\rho]_{i}(\vc[y]) \, dS_y \, . \label{Srhoext} \eeq

Letting $|\Gamma_i| = \int_{\Gamma_i} dS_y$ denote the area of the $i$th boundary,
it is easy to check that $\vc[\rho]_{i} = \vc[F]_{i}/|\Gamma_i|$ produces a net force of
$\vc[F]_{i}$ on $\Gamma_i$ with zero torque.
Likewise, letting $\vc[\rho]_{i} =  \vc[\tau]_{i}^{-1} {\vc[T]_{i}} \times (\vc[x]-\vc[x]^c_i)$ produces zero net forces on each boundary but net torque of $\vc[T]_{i}$ on $\Gamma_{i}$ assuming
$\vc[\tau]_i$ is the moment of inertia tensor:
\beq
{\footnotesize
\vc[\tau]_{i} = 
\begin{bmatrix}  \int_{\Gamma_{i}} \left(x_{2} - x^{c}_{i,2} \right)^2 + \left(x_{3} - x^{c}_{i,3} \right)^2 dS & 
-\int_{\Gamma_{i}} \left(x_{1} - x^{c}_{i,1} \right)\left(x_{2} - x^{c}_{i,2} \right) dS 
& -\int_{\Gamma_{i}} \left(x_{1} - x^{c}_{i,1} \right)\left(x_{3} - x^{c}_{i,3} \right) dS  \\ 
-\int_{\Gamma_{i}} \left(x_{1} - x^{c}_{i,1} \right)\left(x_{2} - x^{c}_{i,2} \right) dS &
\int_{\Gamma_{i}} \left(x_{1} - x^{c}_{i,1} \right)^2 + \left(x_{3} - x^{c}_{i,3} \right)^2 dS & 
 -\int_{\Gamma_{i}} \left(x_{2} - x^{c}_{i,2} \right)\left(x_{3} - x^{c}_{i,3} \right) dS  \\
 -\int_{\Gamma_{i}} \left(x_{1} - x^{c}_{i,1} \right)\left(x_{3} - x^{c}_{i,3} \right) dS &
 -\int_{\Gamma_{i}} \left(x_{2} - x^{c}_{i,2} \right)\left(x_{3} - x^{c}_{i,3} \right) dS  &
 \int_{\Gamma_{i}} \left(x_{1} - x^{c}_{i,1} \right)^2 + \left(x_{2} - x^{c}_{i,2} \right)^2 dS \\
\end{bmatrix}
} \, ,\eeq
where $\vc[x] = \left(x_{1} , x_{2}, x_{3} \right)$  and $\vc[x]^c_{i} = \left(x^c_{i,1}, x^c_{i,2}, x^c_{i,3} \right)$.
Thus, 
\beq \vc[\rho]_i = \frac{\vc[F]_i}{|\Gamma_i|} + \vc[\tau]_{i}^{-1} \vc[T]_i \times (\vc[x]-\vc[x]^c_i)  \label{eq:rhoeq} \eeq
is a force density that satisfies both conditions. 

We must now find a field $\vc[u]_{sc}$ that, while producing zero net forces and torques, enforces that the total velocity  $\vc[u] = \vc[u]_{inc}  + \vc[u]_{sc}$ is a rigid body motion on each surface $\Gamma_i$. If we define $\vc[u]_{sc} = \cA[S]_{\Gamma}[\vc[\mu]]$, we must ensure that: 
\beq \vc[0] = \int_{\Gamma_{i}} \vc[\mu]_{i}(\vc[y]) \, dS_y \, , \ \ \vc[0] = \int_{\Gamma_i} (\vc[y] - \vc[x]^c_i) \times \vc[\mu]_{i}(\vc[y]) \, dS_y \label{eq:Smuzero} \, , \eeq
where $\vc[\mu]_{i} = \vc[\mu] |_{\Gamma_i}$. 

\subsection{Formulation as a fluid stress problem}
In order to ensure that $\vc[u]$ is a rigid body motion, we use the fact that rigid bodies cannot have internal stresses, meaning that traction forces from the interior must be identically equal to zero. This sets the interior boundary value problem for the Stokes 
equation to be $\vc[f]^{-} \equiv 0$ on $\Gamma$. Then, applying the jump relations in \pr{eq:jumptraction} to $\vc[u] = \cA[S]_{\Gamma}[\vc[\rho]+\vc[\mu]]$, we obtain
\beq \left(\frac{1}{2} \cA[I]  + \cA[K]\right)[\vc[\mu]](\vc[x]) = -\left(\frac{1}{2} \cA[I]  + \cA[K]\right)[\vc[\rho]](\vc[x]) \ \ \forall \vc[x] \in \Gamma \label{eq:rigidbodymu} \, ,
\eeq 
a Fredholm equation of the second kind for the unknown density $\vc[\mu]$, where
\[
\cA[I] =\left(\begin{array}{cccc}
\cA[I]_{1}\\
 & \cA[I]_{2}\\
 &  & \ddots\\
 &  &  & \cA[I]_{N}
\end{array}\right)\, , \ \ 
\cA[K] =\left(\begin{array}{cccc}
\cA[K]_{1,1} & \cA[K]_{1,2} & \ldots & \cA[K]_{1,N}\\
\cA[K]_{2,1} & \cA[K]_{2,2} & \ldots & \cA[K]_{2,N}\\
\vdots & \vdots & \ddots & \vdots\\
\cA[K]_{N,1} & \cA[K]_{N,2} &  & \cA[K]_{N,N}
\end{array}\right).
\]
Here,
$\cA[I]_{i}: \cA[X]_{i} \to  \cA[X]_{i}$
is the identity map, 
and $\cA[K]_{i,j}:\cA[X]_{j} \to \cA[X]_{i}$ 
is the operator given by
\beq
\left(\cA[K]_{i,j}\vc[\rho] \right)_{k} = n_l (\vc[x]) \int_{\Gamma_{i}} T_{k,l,m}(\vc[x],\vc[y]) \rho_{m}(\vc[y]) \, dS_y \quad \vc[x]\in\Gamma_{i}
\eeq
for $i \neq j$ and 
\beq
\left(\cA[K]_{i,i}\vc[\rho] \right)_{k} = n_l (\vc[x]) \dashint_{\Gamma_{i}} T_{k,l,m}(\vc[x],\vc[y]) \rho_{m}(\vc[y]) \, dS_y \quad \vc[x]\in\Gamma_{i} \, .
\eeq
and $\cA[X]_{i} = C^{0,\alpha} \left(\Gamma_{i}\right) \times C^{0,\alpha} \left(\Gamma_{i}\right) \times C^{0,\alpha} \left(\Gamma_{i}\right)$, where
$C^{0,\alpha} \left(\Gamma_{i}\right)$ is the H\"{o}lder space with exponent $\alpha>0$.
The operator $\frac{1}{2}\cA[I]+\cA[K]$ has a $6n$ dimensional nullspace in 3 dimensions, corresponding precisely to rotations and translations. 
The integral constraints in \pr{eq:Smuzero} determine the unique solution such that the scattered field does not add net forces or torques. 

\begin{thm}
If $\vc[\mu]\left(\vc[x]\right)$ solves equation \eqref{eq:rigidbodymu},
together with the constraints
\eqref{eq:Smuzero} 
then $\vc[u]\left(\vc[x]\right) = \cA[S]_{\Gamma}[\vc[\rho]+\vc[\mu]]$ solves the mobility problem.
\end{thm}
\begin{proof}
$\vc[u]\left(\vc[x]\right)$ clearly satisfies the Stokes equations in
$E$ by construction.
Using equations \eqref{eq:Smuint} and  \eqref{eq:Smuext},
the choice of $\vc[\rho]$ in equation \eqref{eq:rhoeq}, and the constraints
\eqref{eq:Smuzero},  
we see that $\vc[u]\left(\vc[x]\right)$ satisfies the net force and 
torque conditions 
\eqref{eq:ForceTorque}.
Furthermore, from \eqref{eq:Sgrowthinf}, it follows that 
$\left|\vc[u]\left(\vc[x]\right)\right|\to0$
as $\left|\vc[x]\right|\to\infty$. 
Since $\vc[u]\left(\vc[x]\right)$ solves the Stokes equations
in $D_{i}$ and satisfies $\vc[f]^{-}\equiv0$
on $\Gamma_{i}$, $\vc[u]$ must be a rigid body motion. 
By the continuity of the single layer potential, 
$\vc[u]$ must define a rigid body motion from the exterior as well.
\end{proof}

A standard approach to find $\vc[\mu]$ would be to discretize this equation as well as the integral constraints, and solve the resulting rectangular linear system. This can be avoided by carefully adding the constraints to the integral equation, as we now show. That is,
instead of (\ref{eq:rigidbodymu}), we consider the equation
\beq \frac{1}{2} \vc[\mu](\vc[x]) + \cA[K][\vc[\mu]](\vc[x]) 
+ \int_{\Gamma_i} \vc[\mu](\vc[y]) dS_y +  \left(\int_{\Gamma_i} (\vc[y]-\vc[x]^c_i) \times \vc[\mu](\vc[y]) dS_y\right) \times (\vc[x]-\vc[x]^c_i) = 
-\frac{1}{2} \vc[\rho](\vc[x]) - \cA[K][\vc[\rho]](\vc[x]) \ \ \forall \vc[x] \in \Gamma_i  \label{eq:augsysmu} \eeq  
or
\beq \left(\frac{1}{2}\cA[I]+\cA[K+L] \right) \vc[\mu] = -\left(\frac{1}{2}\cA[I]+\cA[K] \right)\vc[\rho] \label{eq:augsysop} \eeq
where
\beq
\cA[L] =\left(\begin{array}{cccc}
\cA[L]_{1}\\
 & \cA[L]_{2}\\
 &  & \ddots\\
 &  &  & \cA[L]_{N}
\end{array}\right) \, ,
\eeq
with $\cA[L]_{i}[\vc[\mu]_{i}](\vc[x]) = 
  \int_{\Gamma_i} \vc[\mu]_{i}(\vc[y]) \ dS_y +
\left( \int_{\Gamma_{i}} (\vc[y]-\vc[x]^c_i) \times \vc[\mu]_{i}(\vc[y]) dS_y \right) \times (\vc[x]-\vc[x]^c_i) $. 

The following lemma shows that solving \eqref{eq:augsysop} is equivalent to solving \eqref{eq:rigidbodymu}
with the constraints \eqref{eq:Smuzero}.

\begin{lem}
 If $\vc[\mu]$ solves \eqref{eq:augsysop}, then it solves \eqref{eq:rigidbodymu} and \eqref{eq:Smuzero}.
\end{lem}
\begin{proof}
We notice $\cA[L]_i$ can be written as the product of two operators 
$\cA[L]_i = \cA[B]_i \cA[G]_i$, where $(\vc[F],\vc[T]) = \cA[G]_i [\vc[\mu]_i]$ 
computes net force and torque associated with 
$\cA[S]_{\Gamma_{i}}[\vc[\mu]_i]$, and $\cA[B]_i(\vc[F],\vc[T])$ is the function $\vc[F] +  \vc[T] \times (\vc[x]-\vc[x]^c_i)$ on $\Gamma_i$. 
Let $\cA[G]:\prod_{i}^{n} \cA[X]_{i} \to \mathbb{R}^{6n}$ be the operator given by $\cA[G] \vc[\mu] = \left(\cA[G]_{1} \vc[\mu]_{1},\ldots,\cA[G]_{n} \vc[\mu]_{n}\right)$.
Using properties of the traction kernel, one can show that $\cA[G] (\frac{1}{2}\cA[I]+\cA[K]) = 0$. This means, applying $\cA[G]$ on both sides of \pr{eq:augsysop}

\beq \cA[G] \cA[L] \vc[\mu] = 0 \eeq  

By the same reasoning we used to construct the incident field (\pr{eq:rhoeq}), $\vc[\rho] = \cA[L] \vc[\mu]$ 
is a density such that $\cA[S]\vc[\rho]$ has net force $|\Gamma_i| \int_{\Gamma_i} \vc[\mu]_i(\vc[y]) dS_y$ 
and net torque $\vc[\tau]_{i} \int_{\Gamma_i} (\vc[y]-\vc[x]^c_i) \times \vc[\mu]_{i}(\vc[y]) dS_y$ on $\Gamma_i$. Thus,
{\small \beq  \cA[G] \cA[L] \vc[\mu] = \left(|\Gamma_1| \int_{\Gamma_1} \vc[\mu]_1(\vc[y]) dS_y, \vc[\tau]_{1} \int_{\Gamma_1} (\vc[y]-\vc[x]^c_1) \times \vc[\mu]_{1}(\vc[y]) dS_y 
\ldots, |\Gamma_n| \int_{\Gamma_n} \vc[\mu]_n(\vc[y]) dS_y, \vc[\tau]_{n} \int_{\Gamma_n} (\vc[y]-\vc[x]^c_n) \times \vc[\mu]_{n}(\vc[y]) dS_y \right)\eeq}

Hence, this means $\cA[G]_{i} \vc[\mu]_{i} = 0$ (the integral constraints are satisfied), and $\cA[L] \vc[\mu] = 0$, which implies \pr{eq:rigidbodymu} is satisfied. 
\end{proof}

The following lemma shows that the operator $\frac{1}{2}\cA[I]+\cA[K]+\cA[L]$ has no null space.
\begin{lem}
 The operator $\frac{1}{2}\cA[I]+\cA[K]+\cA[L]$ is injective.
\end{lem}
\begin{proof}
 Let $\vc[\mu] \in \mathcal{N}\left(\frac{1}{2}\cA[I] + \cA[K] +\cA[L] \right)$, i.e. it solves
 $\left(\frac{1}{2}\cA[I]+\cA[K]+\cA[L]\right)\vc[\mu] = 0$. 
 Following the reasoning above,
 we conclude that $\vc[\mu]$ satisfies the force and torque constraints given by equation 
 \eqref{eq:ForceTorque}.
 Thus $\cA[L]\vc[\mu] = 0$ and $\left(\frac{1}{2}\cA[I] + \cA[K]\right)\vc[\mu] = 0$. 
 Let $\vc[u]=\cA[S]_{\Gamma}[\vc[\mu]](\vc[x])$.  
 Let $\vc[f]^{-}$ and $\vc[f]^{+}$ denote the interior 
 and exterior limits of the surface traction corresponding to 
 the velocity field $\vc[u]$, respectively. 
 From the properties of the Stokes single layer potential 
 \begin{equation}
 \vc[f]^{-} = \left(\frac{1}{2}\cA[I]+\cA[K] \right)\vc[\mu] = 0 \, .
 \end{equation}
 By uniqueness of solutions to interior surface traction problem, 
 we conclude that $\vc[u]$ is a rigid body motion on each boundary
 component. Thus, $\vc[u]$ solves the mobility problem with $\vc[F]_{i} = \mathbf{0}$ and $\vc[T]_{i}=0$.  
 By uniqueness of
 solutions to the mobility problem, we conclude that $\vc[u]\equiv \mathbf{0}$ in $E$. Hence, 
 $\vc[f]^{+} =0$. From the properties of the Stokes single layer,
 \begin{equation}
  \vc[\mu] = \vc[f]^{-} - \vc[f]^{+}= \mathbf{0} \, .
 \end{equation}
 Therefore, $\mathcal{N} \left(\frac{1}{2}\cA[I] + \cA[K] + \cA[L]\right) = \left \{ 0 \right \}$.
\end{proof}
By the Fredholm alternative, therefore, \eqref{eq:augsysop} has a unique
solution $\vc[\mu]$.

\section{Numerical method}
\label{sec:discretization}
Given $\{\Gamma_i, \vc[F]_i, \vc[T]_i \}_{i=1}^n$, a set of particle boundaries, external forces and torques acting on them respectively, our integral equation formulation to compute the resultant velocity field at any point in the fluid domain or on the particle boundaries can be summarized from the previous section as:
\begin{enumerate}
\item Evaluate the densities $\{\rho_i\}_{i=1}^n$ using \eqref{eq:rhoeq},
\item Solve the system of integral equations \eqref{eq:augsysop} for the unknown densities $\{\mu_i\}_{i=1}^n$.
\item At any $\vc[x] \in \Gamma \cup E$, evaluate the velocity as $\vc[u](\vc[x]) = \cA[S]_{\Gamma}[\vc[\rho] + \vc[\mu]](\vc[x])$. 
\end{enumerate}
We describe our numerical method for discretizing the integral operators and evolving the boundary positions in this section.    

We use spherical harmonic approximations \cite{boyd2001chebyshev} to represent the particle boundaries and the force densities defined on them. The coordinate functions $\vc[x](\theta, \phi)$---where   $\theta$ is the polar angle and $\phi$ is the azimuthal angle---of a particle boundary, for instance, are approximated by their truncated spherical harmonic expansion of degree $p$: 
\begin{equation}
 \vc[x] (\theta,\phi) = 
\sum_{n=0}^p\sum_{m=-n}^n \vc[x]_n^m \, Y_n^m(\theta,\phi), \quad  \theta \in [0,\pi],  \quad \phi \in [0,2\pi].
\label{shesimple_trunc}
\end{equation}
Here, $Y_n^m$ is a spherical harmonic of degree $n$ and order $m$ defined in terms of the associated Legendre functions $P_n^m$
by
\begin{equation}
Y_n^m(\theta,\phi) =
\sqrt{\frac{2n+1}{4\pi}} \,
\sqrt{\frac{(n-|m|)!}{(n+|m|)!}} \, P_n^{|m|}(\cos \theta)\,
		      e^{i m \phi}, \label{ynmpnm}
\end{equation}
and each $\vc[x]_n^m$, a $3\times 1$ vector, is a spherical harmonic coefficient of $\vc[x]$. The finite-term spherical harmonic approximations, such as (\ref{shesimple_trunc}), are spectrally convergent with $p$ for smooth functions \cite{orszag1974fourier}.  
The forward and inverse spherical harmonic transforms \cite{mohlenkamp1999fast} can be used to switch from physical to spectral domain and vice-versa. A standard choice for the numerical integration scheme required for computing these transforms is to use the trapezoidal rule in the azimuthal direction and the Gauss-Legendre quadrature in the polar direction. The resulting grid points in the parametric domain are given by
\beq
\left\{\theta_j = \cos^{-1}(t_j), \, j = 0, \ldots p \right\}, \quad  \text{and} \quad
 \left\{\phi_k =\frac{2 \pi k}{2p+2}, \, k = 0, \ldots, 2p+1 \right\},
 \label{eqn:nodes}\eeq 
where $t_j$'s are the nodes of the $(p+1)$-point Gauss-Legendre quadrature on $[-1, 1]$. Then, the following quadrature rule for smooth integrands is spectrally convergent (as shown in \cite{veerapaneni2011fast}):
\begin{align}
\int_{\Gamma_i} \vc[\mu](\vc[y]) dS_y &= \int_0^{2\pi} \int_0^\pi \vc[\mu](\vc[y](\theta, \phi)) W(\theta, \phi) \, d\theta d\phi, \\[0.5ex]
&=  \sum_{j = 0}^{p} \sum_{k=0}^{2p+1} \frac{2\pi \lambda_j}{(2p+2) \sin \theta_j} 
\vc[\mu](\vc[y](\theta_j, \phi_k)) W(\theta_j, \phi_k), 
\end{align}
where $\lambda_j$'s are the Gauss-Legendre quadrature weights. The area element $W$ is calculated using the standard formula
\begin{equation}
W = \sqrt{(\vc[x]_\theta \cdot \vc[x]_\theta)(\vc[x]_\phi \cdot \vc[x]_\phi) - (\vc[x]_\theta \cdot \vc[x]_\phi)^2},
\label{forms}
\end{equation}
and the derivatives of the coordinate functions $\vc[x]$ are computed via spectral differentiation, that is, 
\begin{equation}
\vc[x]_\theta(\theta_j, \phi_k) = 
\sum_{n=0}^p\sum_{m=-n}^n \vc[x]_n^m \, (Y_n^m(\theta_j,\phi_k))_\theta,
\quad
\vc[x]_\phi(\theta_j, \phi_k) = 
\sum_{n=0}^p\sum_{m=-n}^n \vc[x]_n^m \, (Y_n^m(\theta_j,\phi_k))_\phi.
\label{diff}
\end{equation}
The off-boundary evaluation of all the layer potentials is computed using this smooth quadrature rule. If a target point lies close to the boundary, the layer potentials become {\em nearly-singular} and specialized quadrature rules are required to attain uniform convergence. While several recent works devised high-order schemes for near singular integrals in two dimensions \cite{beale2001method,quaife2014high,helsing2008evaluation,ojala2015accurate,barnett2015spectrally,klockner2013QBX},  limited work exists for three-dimensional problems \cite{ying2006high,af2014fast,tlupova2013nearly}. In this work, we simply upsample the boundary data to evaluate layer potentials at targets that are close. 

Lastly, the fast spherical grid rotation algorithm introduced in \cite{gimbutas2013fast} is used to compute the weakly-singular integral operators, such as $\cA[K]_{i,i}[\cdot]$. We refer to Theorem 1 of \cite{gimbutas2013fast} for the quadrature rule, which is based on the idea that if the spherical grid is rotated so that the target point becomes the north (or south) pole, the integrand transforms into a smooth function. This scheme is spectrally accurate and has a compulational complexity of $\mathcal{O}(p^4 \log p)$. 

\subsection{Fast evaluation of integral operators}
The two integral kernels we must evaluate at each step are $\cA[S]_{\Gamma}$ and $\cA[K]_{\Gamma}$. After discretizing them, their block-diagonal components involve singular integrals in each $\Gamma_i$, which we can compute using the method of \cite{gimbutas2013fast}. However, instead of computing them at every time step, we notice that both kernels are translation invariant, and that although not rotation invariant, for both it is true that if $\Gamma = R \Gamma^0$, where $R$ is a rotation, then 
\beq \cA[S]_{\Gamma} \vc[\mu] = R \cA[S]_{\Gamma^0} R^{*} \vc[\mu]. \eeq
It is thus possible to compute diagonal blocks using the fast singular quadrature only once, and update them at each time step using the corresponding rotation matrices $R_i(t)$. 

Interactions between different surfaces must of course be evaluated at each time step. When the system size $N = 6p(p+1)n$ is large, we use the Stokes 3D FMM library, \emph{STKFMMLIB3D} \cite{toolboxfmm}, to accelerate this evaluation. In order to compute the traction kernel $\cA[K]_{\Gamma}$, we turn on the \texttt{ifgrad} flag in the Stokes particle FMM routine of this library to compute both the pressure $p$ and the gradient of the single layer potential, $\vc[u] = \cA[S]_{\Gamma} \vc[\rho]$. The application of the traction kernel is then computed at target points as 
\begin{equation}
\cA[K]_{\Gamma}\vc[\rho] = [-p I + \nabla \vc[u] + \nabla \vc[u]^{T} ] \vc[n]. 
\end{equation}

\subsection{Fast solution of integral equations}

For large system size $N$, we employ the fast evaluation scheme and the Krylov subspace iterative method GMRES \cite{GMRES} to solve for the scattered field force density $\vc[\mu]$ in \eqref{eq:augsysop}. This system of integral equations is of second kind, and as such is generally well-conditioned. The number of GMRES iterations is dependent on the system matrix eigenspectrum, and as shown in Section \ref{sec:numerical}, this number typically varies with geometric complexity of $\Gamma$ and the distance between surfaces (increases as surfaces come close to touching). 

For the examples presented in this work, we accelerate this iterative method by using a simple block-diagonal preconditioner, corresponding to the inverse of self-interactions for each surface $\Gamma_i$. The block-diagonal components for the preconditioner are also computed only once, and updated using the corresponding rotation matrices $R(t)$. For cases in which a more robust preconditioner is warranted, sparse approximate inverse (SPAI) \cite{vavasis1992preconditioning} or multi-level hierarchical factorization methods \cite{bebendorf2005hierarchical,quaife2015preconditioners,ho2013hierarchical_ie,coulier2015inverse,CMZ14,corona2015tensor} are recommended.  

\subsection{Time-stepping scheme} Since $\Gamma_i$ is rigid, we can represent the position of a point $\vc[X](t) \in \Gamma_i^t$ at time $t$ as 
\beq \vc[X](t) = \vc[X]^c_i(t) + R_i(t) (\vc[X](0) - \vc[X]^c_i(0)), \eeq
where $\vc[X]^c_i(t)$ is the centroid of $\Gamma_i^t$ and $R_i(t)$ is a rotation matrix. Although it is possible to use $\vc[u]$ to evolve $\vc[X]$ directly, we want to avoid any distortion of the shape of $\Gamma_i$. Instead, we compute $\vc[v]_i$ and $\vc[\omega]_i$, and evolve $\vc[X]^c_i(t)$ and $R_i(t)$.

To find $\vc[v]_i$, we compute the average of $\vc[u]$ on $\Gamma_i$: 
\beq \vc[v]_i = \frac{1}{|\Gamma_i|} \int_{\Gamma_i} \vc[u](y) dS_y. \eeq
We then set $\vc[\omega]_i \times (\vc[x] - \vc[x]^c_i) = \vc[u]-\vc[v]_i$ and solve for $\vc[\omega]_i$ as
\beq \vc[\omega]_i = \vc[\tau]_{i}^{-1} \int_{\Gamma_i} (\vc[y] - \vc[x]_i^c) \times \vc[u](\vc[y]) dS_y \, .\eeq  
The centroid positions and the rotation matrices are then evolved using the equations
\beq \vc[\dot{X}]^c_i(t) = \vc[v]_i(t,\vc[X](t)) \ \ \ \ \ \dot{R}_i(t) \vc[z] = \vc[\omega]_i(t,\vc[X](t)) \times \vc[z] = M({\vc[\omega]_i(t,\vc[X](t))})\vc[z], \label{eq:RvODEs} \eeq
\beqn \text{where} \qquad M(\vc[\omega]_i) = \bbm 0 & -\omega_3 & \omega_2 \\ \omega_3 & 0 & -\omega_1 \\ -\omega_2 & \omega_1 & 0\ebm \hspace{-0.05in}. \eeqn

For constant $\vc[\omega]$, the solution to the rotation matrix in \pr{eq:RvODEs} is $R(t) = e^{Mt}$, which we can evaluate as $I + \sin t M + (1-\cos t)M^2$ using the Rodrigues rotation formula. In general, the solution can be written as $R(t) = \mathrm{exp}(\int_0^t M(\vc[\omega](t)) dt)$. 
We then discretize the system \eqref{eq:RvODEs} in time using an explicit Runge-Kutta method and evolve centroids and rotation matrices accordingly. We note that, in order to preserve orthogonality over time, it is sometimes desirable to implement an equivalent integration for the rotation matrix $R(t)$ using a corresponding unit quaternion $q(t)$.  

\subsection{Contact force algorithm}
In the evolution of rigid bodies in fluid flow, it is crucial to avoid collisions and overlap between surfaces. We employ the contact algorithm described in \cite{das2013ehd}: for each pair of surfaces $\Gamma_1,\Gamma_2$ touching at a point $x^*$, we introduce contact forces $\vc[F]_1,\vc[F]_2$ such that $\vc[F]_1 = -\vc[F]_2$ and the corresponding velocities satisfy a no-slip condition at $x^*$: 

\begin{equation}
\vc[v]_1 + \vc[\omega]_1 \times (x^* - x_1^c) = \vc[v]_2 + \vc[\omega]_2 \times (x^* - x_2^c).
\end{equation}
Let $\vc[F]_c$ be an array containing the $n_c$ contact forces (only one force per pair is needed, since forces applied to each surface are opposite). Given a configuration of rigid bodies with $n_c$ pairs in contact, let $C$ be the sparse, $N \times 3 n_c$ array such that the corresponding array of forces and torques applied to $\Gamma$ is $\vc[F] = C \vc[F]_c$. Let $D$ then be the $3 n_c \times N$ array such that, given translational and rotational velocities in $\vc[V]$, $D \vc[V]$ computes the array of differences between velocities at contact points. 

Given an initial set of rigid body velocities $\vc[V]_0$, the additional contact forces needed to enforce no-slip conditions at contact points satisfy
\begin{equation}
(DMC) \vc[F]_c = -D \vc[V]_0, 
\label{eq:eq_contact_force}
\end{equation}
where $M$ is the mobility matrix. Our contact algorithm proceeds by constructing the matrix $DMC$ (corresponding to solving $3 n_c$ instances of the mobility problem) and solving for $\vc[F]_c$. The corresponding updates incorporating the effect of contact forces on $\vc[\mu]$, $\vc[\sigma]$, $\vc[u]$ and $\vc[V]$ are then computed. 

\subsection{Summary of the method}

Given an initial configuration $\Gamma^0$ and prescribed forces and torques $\vc[F]_i(t)$, $\vc[T]_i(t)$, we want to find $\vc[\rho](\vc[x],t)$, $\vc[\mu](\vc[x],t)$, the centroids $\vc[X]^c_i(t)$ and rotation matrices $R_i(t)$ at times $t_k = k \Delta t$ for $k = 1,\dots,m$.  

\begin{algorithm}[H]
\begin{center}
\begin{algorithmic}[1]
\vspace{0.1in}
\STATE Evaluate single layer $\cA[S]_{\Gamma^0}$ and traction $\cA[K]_{\Gamma^0}$ kernels. 
\FOR{ $k = 0 : m$}
\STATE Compute incident field density:~\label{lin:inc} \beqn \vc[\rho]^k_i(\vc[x]) = \frac{\vc[F]_i(t_k)}{|\Gamma_i|} + \vc[\tau]_{i}^{-1}\vc[T]_i(t_k)\times (\vc[x]-\vc[x]^c_i) \eeqn
\STATE Solve for scattered field density~\label{lin:solve} \eqref{eq:augsysop}:  \beqn \left(\frac{1}{2}\cA[I]+\cA[K]+\cA[L] \right) \vc[\mu]^k = -\left(\frac{1}{2}\cA[I]+\cA[K] \right)\vc[\rho]^k \eeqn
\STATE Evaluate $\vc[u] = \cA[S]_{\Gamma^k} [ \vc[\rho]^k + \vc[\mu]^k ]$ and compute rigid body velocities $(\vc[v]_i(t) ,\vc[\omega]_i (t) )$.~\label{lin:eval} 
\IF{Contact occurs}
\STATE Compute contact force $\vc[F]_c$ satisfying \eqref{eq:eq_contact_force}.~\label{lin:contact} 
\STATE Update $\vc[\rho]^k$ and $\vc[\mu]^k$, evaluate $\vc[u]$ and update $(\vc[v]_i(t) ,\vc[\omega]_i (t) )$
\ENDIF
\STATE Evolve centroids and rotation matrices to obtain $\vc[X]^c_i(t_{k+1})$, $R_i(t_{k+1})$
\STATE Update diagonal blocks of $\cA[S]_{\Gamma^k}$ and $\cA[K]$.~\label{lin:opdiag} 
\STATE Evaluate off-diagonal interactions.~\label{lin:opoffd} 
\ENDFOR 
\end{algorithmic}
\end{center}
\caption{Rigid body Stokes evolution}
\label{alg:RBS}
\end{algorithm}

\section{Numerical experiments}
\label{sec:numerical}

We present three representative applications of our boundary integral equation method, specifying in each case the context for the mobility problem formulation. Forces and torques applied to each particle surface $\Gamma_i$ are introduced via the incident surface force density $\rho_i$. We then perform a series of experiments to test performance and scaling of the proposed solver. 

\subsection{Experimental setup}

\subsubsection{Sedimentation of particles under gravity}
Suspension of rigid particles sedimenting under the action of gravity in viscous flows are commonly observed in many natural and engineering systems. The particle dynamics can be significantly complex even for simpler problem setups, especially for non-spherical particles (see \cite{af2014fast} and references therein). This problem can simply be viewed as an instance of the mobility problem with external forces and torques given by, 
\beq \vc[F]_{i} = \Delta m_i \vc[g], \quad \vc[T]_i = \vc[0], \quad i =1, \ldots, n, \label{eq:gravity} \eeq
where $\vc[g]$ is the gravity vector and $\Delta m_i$ is the difference between the mass of the $i^{\text{th}}$ particle and the surrounding fluid of same volume. Therefore, we can use the formulation developed in this work and, from \eqref{eq:gravity}, evaluate the scattered and incident field densities. Using the numerical algorithm discussed in Section \ref{sec:discretization}, we solved for the evolution of various particle configurations under gravity; Figure \ref{fig:gravity} shows two examples. 

\begin{figure}[H]
\centering
\includegraphics[width=\textwidth]{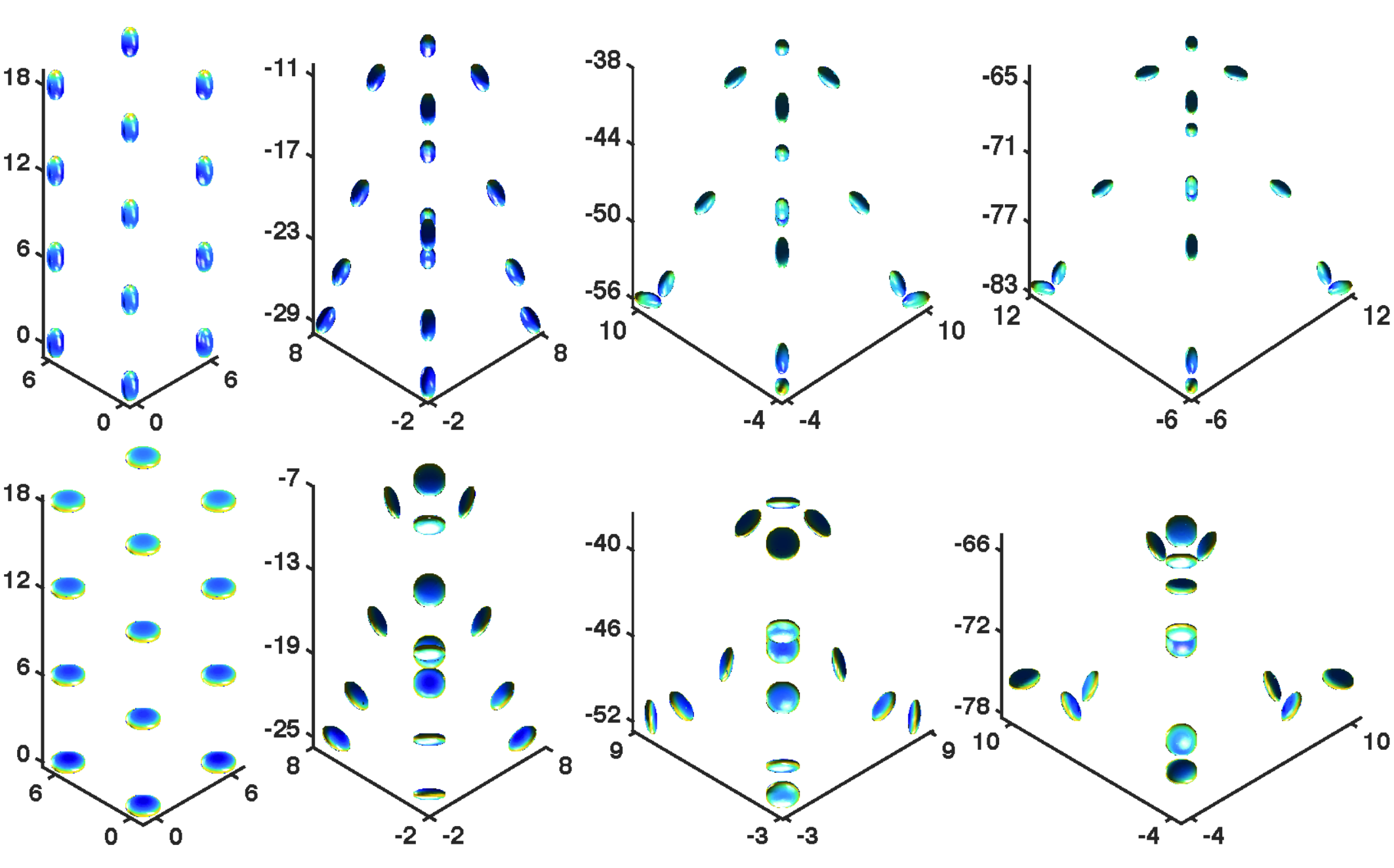}
\mcaption{fig:MHD_Gravity_prolatevsoblate}{Sedimentation flow - regular lattice}{Simulation of sedimentation for two lattices of ellipsoidal objects of equal weight. Color scheme corresponds to the magnitude of the total force field $||\rho_i + \mu_i||$. \emph{Top:} snapshots for sedimentation of prolate ellipsoids, with semi-axes $(\frac{1}{2},\frac{1}{2},1)$. \emph{Bottom:} snapshots for sedimentation of oblate ellipsoids, with semi-axes $(1,1,\frac{1}{3})$.}
\label{fig:gravity}\end{figure}

\subsubsection{Low-Re swimmer models}
One of the archetypal physical models to study swimming at low Reynolds numbers is the Najafi-Golestanian model \cite{najafi2004simple}, comprised of three aligned spheres connected by two arms. Because of the time-reversal symmetry of the Stokes equations, these arms must undergo a series of deformations asymmetric in time in order for the swimmer to locomote. The study and design of swimmers at small scales thus often relies on models that prescribe a series of strokes (displacements) for each arm between linked objects.

Alternatively, exploiting the linearity of the mobility problem (\ref{eq:mobility}), we opt for a model with periodic force and torque prescription for illustration purposes\footnote{Note that in ongoing work we have also developed methods to obtain forces and torques from a set of prescribed strokes using the mobility problem formulation.}. The simplest instance consists of three aligned spheres with prescribed periodic, oscillating forces that sum to zero and no external torques:  
\begin{equation}
\vc[F]_1(t) = (2\cos t + \sin t) \vc[e_1], \, \vc[F]_2(t) = (\sin t - \cos t) \vc[e_1], \ \vc[F]_3(t) = (-\cos t - 2\sin t) \vc[e_1], \quad \left\{\vc[T]_i\right\}_{i=1}^3 = \vc[0].
\label{eq:3-sphere} \end{equation}
Due to the break in symmetry between the net forces for the two consecutive pairs of spheres, this arrangement has a net positive displacement on each time period $t \in [2\pi k , 2\pi (k+1) ]$. Since forces and torques are prescribed, setting up the BIE formulation and its numerical solution proceeds the same way as in Algorithm 1.  We depict the self-locomotion of this 3-sphere swimmer in Figure \ref{fig:3-sphere}. In addition, we use this model to test the convergence properties of our numerical algorithm in Section \ref{sc:convergence}.  

\begin{figure}[H]
\centering
\includegraphics[width=0.3\textwidth]{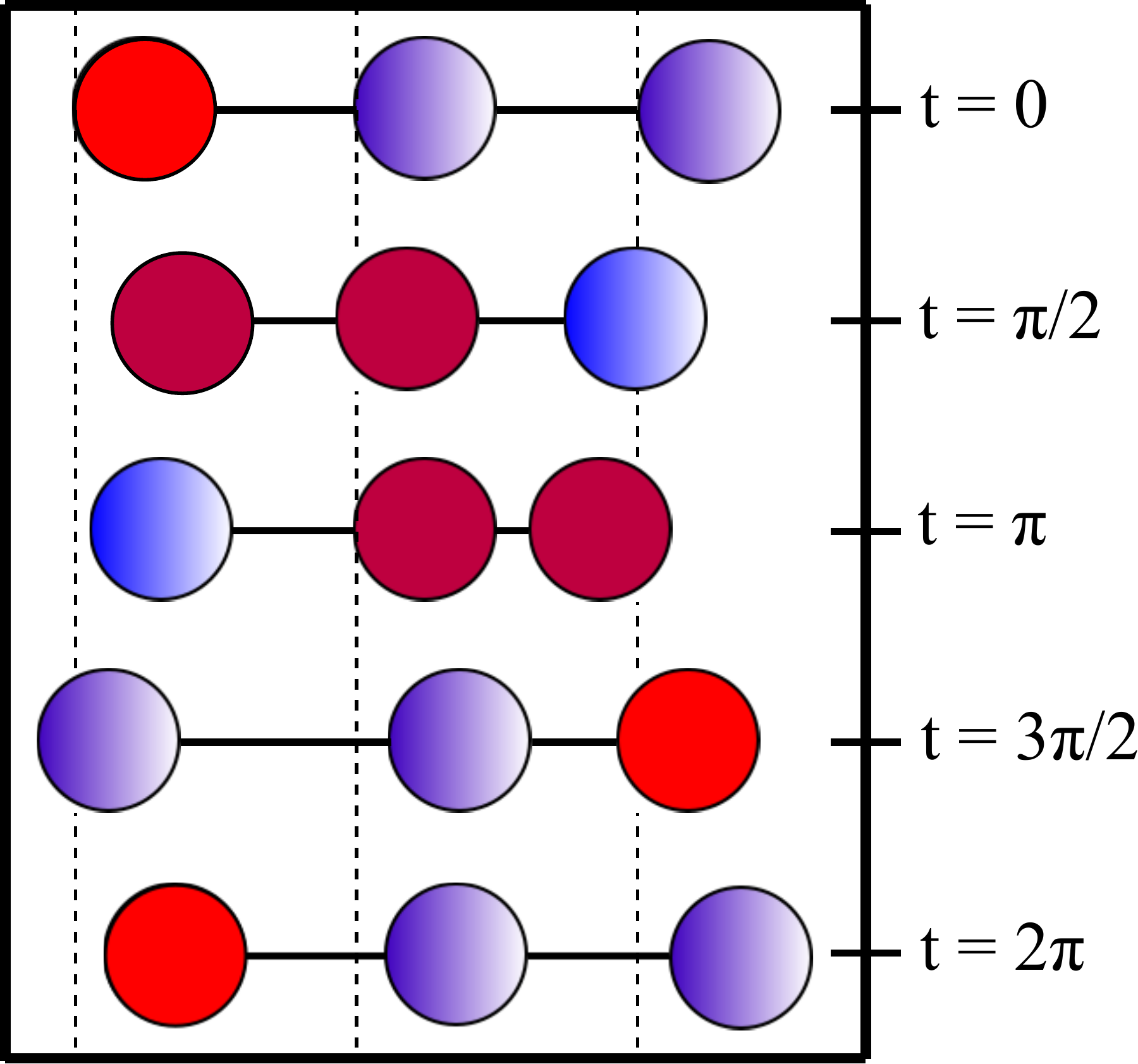}
\mcaption{fig:3sphswim_diagram}{Force based 3-sphere swimmer}{Motion of the 3-sphere swimmer with prescribed forces and torques given in \eqref{eq:3-sphere}. Spheres are colored on a linear gradient according to the force applied, with $F=2$ in red and $F=-2$ in blue. At the end of each cycle, the swimmer has a positive net displacement.}
\label{fig:3-sphere} \end{figure}

\subsubsection{Magnetorheological fluid flows}

A range of problems of interest in smart material design feature the study of colloidal suspensions which, when subjected to uniform magnetic fields \cite{maxeykeaveny} (or electric fields \cite{das2013ehd}), change their apparent viscosity and thus, their response to stress. These fluids are used in smart damping technology, with industrial, military and biomedical applications. 

In order to study the rheology of these suspensions, especially at high concentration, a coupled system of evolution equations need to be solved to capture magnetic and hydrodynamic interactions. We assume the fluid to have zero susceptibility, and an absence of free currents. This allows us to represent the corresponding field as the negative gradient of a scalar potential $\phi$ and reduces the static Maxwell equations to a Laplace equation for $\phi$ with prescribed jump conditions at the particle boundaries. For example, for the magnetostatic case, the potential must satisfy the following equations:

\begin{equation} \Delta \phi (\vc[x]) = 0 \ \ \ \forall \, \vc[x] \notin \Gamma, \label{eq:laplace}
\end{equation}

\begin{equation} [[\phi]]_{\Gamma} = 0, \ \ \ \left[ \left[ \mu \frac{\partial \phi}{\partial \vc[r]} \right] \right]_{\Gamma} = 0, \label{eq:lap_jump_conditions} \end{equation}

\begin{equation} \phi \rightarrow -\vc[H_0] \cdot \vc[r] \ \ \mathrm{as} \ \ |\vc[x]| \rightarrow \infty, \label{eq:laplace_infinity} \end{equation}
where $\mu$ is the magnetic permeability and $\vc[H]_0$ is the imposed magnetic field. We note that standard second-kind integral equation formulations exist for this problem \cite{ly1999MHD}.  If we represent $\phi = -\vc[H]_0 \cdot \vc[x] + \cA[S]^{L}[q](\vc[x])$, where $\cA[S]^{L}$ is the Laplace single layer potential, $\cA[K]^{L}$ its normal derivative and $\eta = \frac{\mu - \mu_0}{\mu + \mu_0}$, enforcing the standard jump conditions \cite{kressbook} leads to the equation: 

\begin{equation} \left(\frac{1}{2} I + \eta \cA[K]^{L}_{\Gamma}\right)[q](x) = \eta \vc[H]_0 \cdot \vc[n]. \label{eq:magnetic} \end{equation}

In order to couple \pr{eq:magnetic} with our integral equation method for the fluid velocity, we recall that since the fluid medium is assumed to be insusceptible, the only interaction between them occurs through the traction force applied to the particle surfaces. For a given configuration of rigid bodies, we obtain the corresponding scalar potential $\phi$, and set the incident force field density $\rho_i$ to be the corresponding traction on surface $\Gamma_i$, which is computed using the Maxwell stress tensor. 

We present a few characteristic examples of magnetorheological flow of suspensions of paramagnetic beads, following \cite{maxeykeaveny}. For two spherical beads, subjecting them to a uniform magnetic field $\vc[H]_0$ produces repulsive forces if $\vc[H]_0$ is perpendicular to the line that crosses their centers (see \pr{fig:MHD_repel}), and attractive forces if it is parallel. In the general case, it has been observed that particles tend to align, forming chains along the lines of magnetic field flux, as in Figures \ref{eyecandy}, \ref{fig:MHD_triangle} and \ref{fig:MHD_chain}. \\

\begin{figure}[H]
\centering
\includegraphics[width=0.8\textwidth]{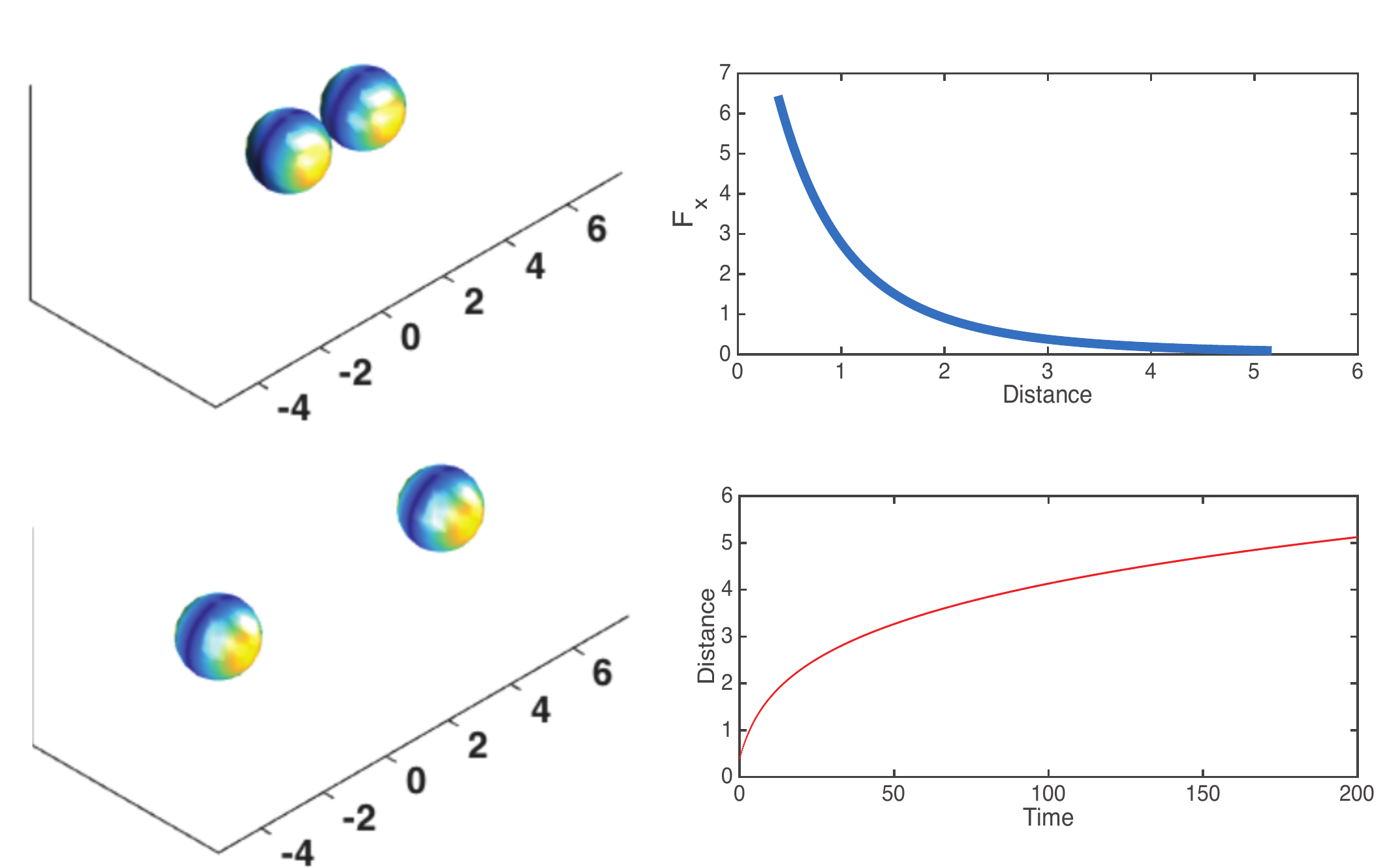}
\mcaption{fig:MHD_repel}{MHD flow - two sphere repulsion}{Simulation of two spheres of radius $1$ in the x-axis, subjected to the uniform field $\vc[H]_0 = (0,10,0)$, with $\mu / \mu_0 = 2$. \emph{Left:} snapshots of simulation at $t=0$ and $t=200$. Color scheme is proportional to magnitude of incident force density $||\rho_i(x)||$. \emph{Top right:} repulsion force magnitude as a function of sphere separation. \emph{Bottom right:} separation as a function of time for $t \in [0,200]$.}
\end{figure}

\begin{figure}[H]
\centering
\includegraphics[width=\textwidth]{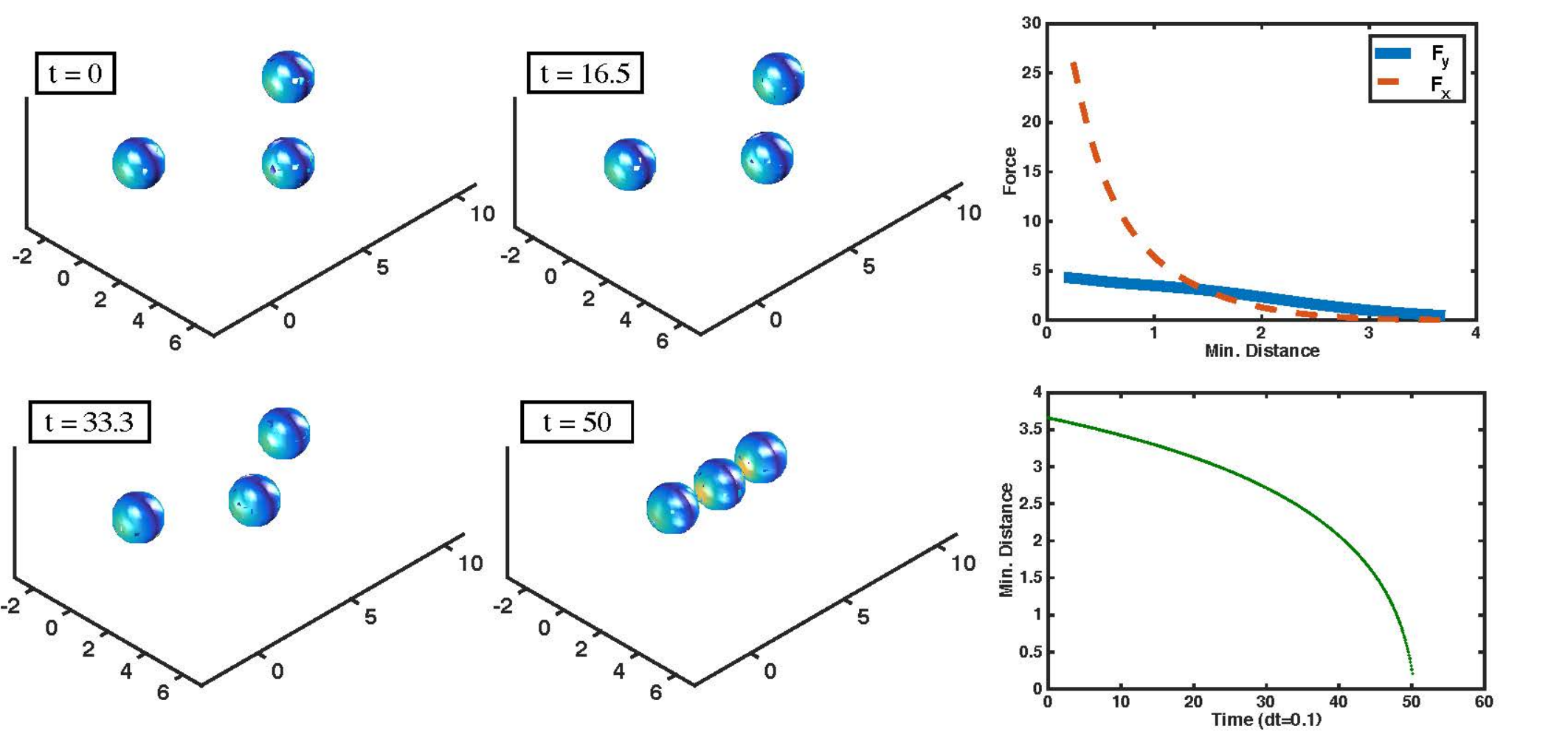}
\mcaption{fig:MHD_triangle}{MHD flow - triangular configuration}{Simulation of a triangular array of spheres of radius $1$, with centers $(0,0,0)$,$(0,8,0)$ and $(4,4,0)$ subjected to the uniform field $\vc[H]_0 = (0,10,0)$, with $\mu / \mu_0 = 2$. \emph{Top:} snapshots at $t=0,16.5,33.3$ and $t=50$. Spheres are colored according to the magnitude of incident force density $||\rho_i(x)||$. \emph{Bottom left:} attractive force magnitude in $x$ and $y$ as a function of minimum sphere separation. \emph{Bottom right:} separation as a function of time $t \in [0,50]$.}
\end{figure}

\begin{figure}[H]
\centering
\includegraphics[width=0.8\textwidth]{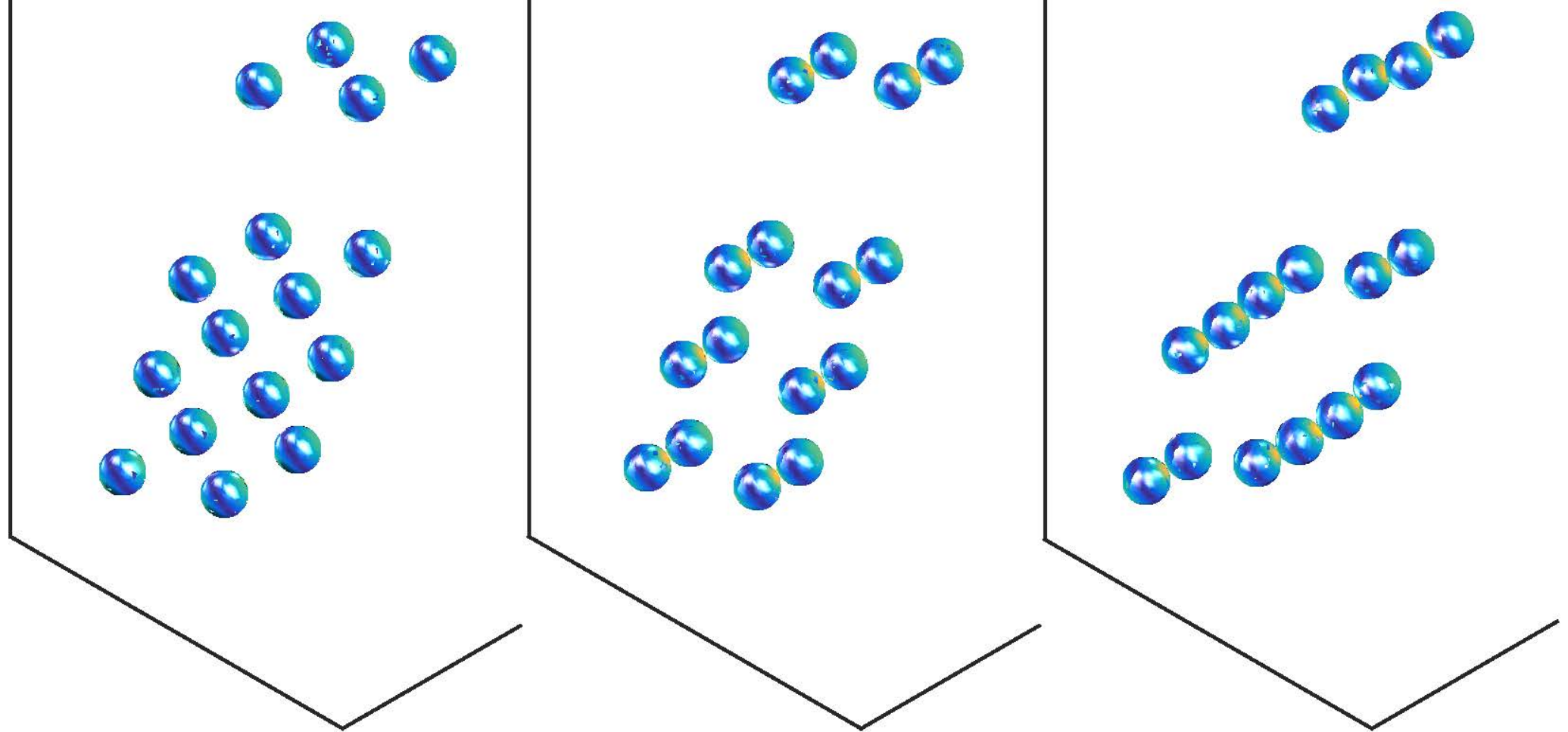}
\mcaption{fig:MHD_chain}{MHD flow - chain formation}{Simulation of spheres of radius $1$ with centers on a perturbed lattice, subjected to the uniform field $\vc[H]_0 = \frac{10}{\sqrt{3}}(1,-1,-1)$, with $\mu / \mu_0 = 2$. Spheres are colored according to the magnitude of incident force density $||\rho_i(x)||$. \emph{Left:} initial configuration. \emph{Middle: } pairs form among closest spheres. \emph{Right:} pairs begin to coalesce into larger chains, aligned with the magnetic field.}
\end{figure}

\subsection{Scaling tests}
\label{sec:scaling}

In order to test experimental scaling for the computational costs of our method, we set a sedimentation simulation with an initial configuration of $2 \times 2 \times n_z$ spheres of radius $1$ on a regular lattice with center spacing equal to $5$. We then evolve the system for $100$ timesteps and record the average time our solver takes per timestep. We test scaling with respect to the number of objects $n=4 n_z$ for $n=\{32,64,128\}$ as well as the order $p$ of spherical harmonic approximations by comparing results for $p=8$ and $p=16$. All tests are performed serially on Intel~Xeon~E--$2690$~v$2$($3.0$~GHz) nodes with $24$~GB of memory.  

\begin{table}[!htb]
  \newcolumntype{M}{>{\centering\let\newline\\\arraybackslash\hspace{0pt}$}c<{$}}
  \newcolumntype{N}[1]{>{\centering\let\newline\\\arraybackslash\hspace{0pt}}D{.}{.}{#1}}
  \centering
  \small
  \begin{tabular}{N{3.0} N{3.2} N{3.2} N{3.2} N{3.2}}\toprule
    \multirow{2}{*}[-5pt]{$n$}
    & \multicolumn{2}{c}{$p=8$}
    & \multicolumn{2}{c}{$p=16$}
     \\ 
    & N & \multicolumn{1}{c}{\parbox[c]{.6in}{\centering Time~(sec)}}
    & N & \multicolumn{1}{c}{\parbox[c]{.6in}{\centering Time~(sec)}}
    \\ \midrule
    32   & 13824  & 12.88  & 52224   & 132.3 \\
    64   & 27648  & 24.97  & 104448 & 272.4  \\
    128 & 55296  & 49.49  & 208896 & 544.2  \\
     \bottomrule
  \end{tabular}
  \mcaption{tbl:scaling}{Sphere lattice total per timestep}{Average timings (in seconds) per timestep for the evolution under gravity of a $2 \times 2 \times n_z$ regular lattice of spheres of radius $1$ and spacing $5$.}
\end{table}

Theoretical scaling for the matrix apply is $\O{n ( p^4+p^2)}$, since it involves the computation of self-interactions (dense apply of $n$ matrices of size $6p(p+1) \times 6p(p+1)$) and one particle Stokes FMM, which is $\O{N}=\O{p^2n}$. As we observe in our profiling, computational cost per timestep is largely dominated by the solution of the integral equation for the scattered field density (using GMRES). Hence, the experimental scaling observed matches that of the fast matrix apply: we observe linear scaling with respect to $n$, and doubling $p$ increases the cost by a factor of $11$ for the cases observed. 

\subsubsection*{Profiling}

For each of the experiments described above, we measure average timings for the steps in \pr{alg:RBS}: incident field density computation (\pr{lin:inc}), scattered field density solve (\pr{lin:solve}), rigid body velocity evaluation (\pr{lin:eval}), collision detection (\pr{lin:contact}), and operator update (\pr{lin:opdiag} - \ref{lin:opoffd}).

We observe that computational cost is largely dominated by the solve, taking up about $80\%$ of the total time per timestep. It takes $4$-$5$ iterations for the preconditioned GMRES to reach target accuracy, which we set at $10^{-6}$. We note that iteration counts and average solve times, while largely independent of $n$ and $p$, tend to increase in cases where surfaces come close to contact. This may be addressed by the use of more adequate preconditioners (such as sparse approximate inverse (SPAI), multilgrid or low accuracy direct solvers). 

In the operator update stage, we update self-interaction matrices as well as the block-diagonal preconditioner to match the new center positions and rotation matrices. This stage takes up $10$-$15\%$ of the total time. The cost of evaluating $\vc[u]$ and computing rigid body velocities is essentially that of a fast matrix-vector apply, accounts for most of the remaining timings observed. 

\begin{table}[!htb]
  \newcolumntype{M}{>{\centering\let\newline\\\arraybackslash\hspace{0pt}$}c<{$}}
  \newcolumntype{N}[1]{>{\centering\let\newline\\\arraybackslash\hspace{0pt}}D{.}{.}{#1}}
  \centering
  \small
  \begin{tabular}{N{3.0} N{3.2} N{3.2} N{3.2} N{3.2} M N{3.2} N{3.2} N{3.2} N{3.2}}\toprule
    \multirow{2}{*}[-5pt]{$n$}
    & \multicolumn{4}{c}{$p=8$} &
    & \multicolumn{4}{c}{$p=16$}
     \\ 
    & N 
    & \multicolumn{1}{c}{\parbox[c]{.6in}{\centering Apply~(sec)}}
    & \multicolumn{1}{c}{\parbox[c]{.6in}{\centering Solve~(sec)}}
    & \multicolumn{1}{c}{\parbox[c]{.6in}{\centering Update~(sec)}} &
    & N 
    & \multicolumn{1}{c}{\parbox[c]{.6in}{\centering Apply~(sec)}}
    & \multicolumn{1}{c}{\parbox[c]{.6in}{\centering Solve~(sec)}}
    & \multicolumn{1}{c}{\parbox[c]{.6in}{\centering Update~(sec)}}
    \\ \midrule
    32   & 13824  & 1.09 & 10.25 & 1.53  & & 52224   & 10.7 & 108.9 & 12.65 \\
    64   & 27648  & 2.26 & 19.51 & 3.19  & & 104448 & 21.5 & 220.3 & 30.49  \\
    128 & 55296  & 4.75 & 38.41 & 6.30  & & 208896 & 42.9 & 439.9 & 61.18  \\
     \bottomrule
  \end{tabular}
  \mcaption{tbl:profiling}{Sphere lattice profile per timestep}{Average timings (in seconds) per timestep for scattered field density solve and operator updates. Matrix-vector apply time is included to account for velocity evaluation, as well as for comparison with the GMRES solve. Incident field density computation is negligible, and no contact force correction was necessary for the experiments presented.}
\end{table}

The same set of tests were ran for regular lattices of ellipsoids with semi-axes $(a,b,c) = (1,0.5,0.5)$ and $(a,b,c) = (1,1,0.5)$. We note that, while showing a slight increase in the solve stage (corresponding to iteration counts of $5$-$6$ iterations for the preconditioned GMRES), their performance in terms of scaling and profiling was similar to that presented above. 

\subsection{Convergence tests}
\label{sc:convergence}

In order to test convergence of our method with respect to the integral operator discretization (controlled by the order $p$ of the spherical harmonic appproximation) and the proposed time-stepping schemes, we set a series of experiments for a swimmer comprised either of three spheres or of three ellipsoids with the same proportions. We apply periodic, oscillatory forces in the $x$-axis, as well as oscillatory torques along the $y$-axis. 

\begin{equation*}
\vc[F]_1(t) = (2\cos t + \sin t) \vc[e_1], \ \vc[F]_2(t) = (\sin t - \cos t) \vc[e_1], \ \vc[F]_3(t) = (-\cos t-2\sin t) \vc[e_1]
\end{equation*}
\begin{equation*}
\vc[T]_1(t) = (2\cos t + \sin t) \vc[e_2], \ \vc[T]_2(t) = (\sin t - \cos t) \vc[e_2], \ \vc[T]_3(t) = (-\cos t - 2 \sin t) \vc[e_2]
\end{equation*}

These modifications to the standard swimmer allow us to test convergence of both centers and rotation matrices, for both spherical and non-spherical shapes. At the end of each cycle of size $2\pi$, the three object swimmer displays net positive displacement along the $x$-axis, as well as slight displacements in the $z$-axis.  

\subsubsection*{Spatial discretization}

We evolve one three-object swimmer for a full cycle $t \in [0,2\pi]$, for $p \in \{ 2,4,8,16\}$, using a forward euler time-stepping scheme with $\Delta t = \frac{2 \pi}{128}$. At the final time $T = 2\pi$, we test self-convergence of the rigid object centers and rotations by measuring the following quantities: 

\begin{equation*} \mathcal{E}_C (T,p) = -\log_2 \underset{i \leq 3}{max} || C_i^p (T) - C_i^{2p}(T) ||_2 \end{equation*}
\begin{equation*} \mathcal{E}_R (T,p) = - \log_2 \underset{i \leq 3}{max} || R_i^p (T) - R_i^{2p}(T) ||_F \end{equation*}

We compare results for spheres and ellipsoids with semi-axes $(a,a,1)$ for $a=\frac{1}{2},\frac{1}{4}$. 

\begin{table}[!htb]
  \newcolumntype{M}{>{\centering\let\newline\\\arraybackslash\hspace{0pt}$}c<{$}}
  \newcolumntype{N}[1]{>{\centering\let\newline\\\arraybackslash\hspace{0pt}}D{.}{.}{#1}}
  \centering
  \small
  \begin{tabular}{N{3.5} N{3.2} N{3.2} N{3.2} N{3.2}}\toprule
    &
    & \multicolumn{1}{c}{\parbox[c]{.6in}{\centering $p=2$}}
    & \multicolumn{1}{c}{\parbox[c]{.6in}{\centering $p=4$}}
    & \multicolumn{1}{c}{\parbox[c]{.6in}{\centering $p=8$}}
    \\ \midrule
    \multirow{2}{*}[-5pt]{Sphere} & \mathcal{E}_C (T,p)  & 5.67 & 14.86 & 29.98   \\
    & \mathcal{E}_M (T,p)  & 9.31 & 15.19 & 29.93   \\ \midrule
    \multirow{2}{*}[-5pt]{Ellipsoid $\frac{1}{2}$} & \mathcal{E}_C (T,p)  & 3.76 & 7.12 & 13.96   \\
    & \mathcal{E}_M (T,p)  & 5.13 & 9.77 & 15.82   \\ \midrule
    \multirow{2}{*}[-5pt]{Ellipsoid $\frac{1}{4}$} & \mathcal{E}_C (T,p)  & -  & 7.53  & 10.83   \\
    & \mathcal{E}_M (T,p)  & - & 7.98  & 10.95    \\
     \bottomrule
  \end{tabular}
  \mcaption{tbl:convp}{Spatial discretization self-convergence}{Error terms for center and rotation matrix evolution of three-object swimmers at final time $T = 2\pi$ for for $p \in \{ 2,4,8,16\}$.}
\end{table}

For both center and rotation matrix evolution we compare the columns in \pr{tbl:convp} and observe that the convergence rate is proportional to $p$. This is in accordance with the theoretical spectral convergence of the quadrature rules employed. We note that, as we increase surface complexity (in this case dependent on the eccentricity of the ellipsoids), the approximation order $p$ required to reach a certain target accuracy increases.   

\subsubsection*{Time discretization}

We test the rate of convergence with respect to $\Delta t$ of our proposed time-stepping scheme for trapezoidal and Runge-Kutta 4 methods. We again evolve a three-object swimmer for a full cycle $t \in [0,2\pi]$, for $p = 8$ and $\Delta t = 2\pi / \{16,32,\dots,256\}$. At the final time $T=2\pi$, we measure the following quantities: 

\begin{equation*} \mathcal{E}_C (T,\Delta t) = -\log_2 \underset{i \leq 3}{max} || C_i^{\Delta t} (T) - C_i^{\Delta t/2}(T) ||_2 \end{equation*}
\begin{equation*} \mathcal{E}_R (T,\Delta t) = - \log_2 \underset{i \leq 3}{max} || R_i^{\Delta t} (T) - R_i^{\Delta t / 2}(T) ||_F \end{equation*}

\begin{table}[!htb]
  \newcolumntype{M}{>{\centering\let\newline\\\arraybackslash\hspace{0pt}$}c<{$}}
  \newcolumntype{N}[1]{>{\centering\let\newline\\\arraybackslash\hspace{0pt}}D{.}{.}{#1}}
  \centering
  \small
  \begin{tabular}{N{3.5} N{3.2} N{3.2} N{3.2} N{3.2} N{3.2}}\toprule
    &
    & \multicolumn{1}{c}{\parbox[c]{.6in}{\centering $\Delta t=\frac{2\pi}{16}$}}
    & \multicolumn{1}{c}{\parbox[c]{.6in}{\centering $\Delta t=\frac{2\pi}{32}$}}
    & \multicolumn{1}{c}{\parbox[c]{.6in}{\centering $\Delta t=\frac{2\pi}{64}$}}
    & \multicolumn{1}{c}{\parbox[c]{.6in}{\centering $\Delta t=\frac{2\pi}{128}$}}
    \\ \midrule
    \multirow{2}{*}[-5pt]{Sphere} & \mathcal{E}_C (T,\Delta t)  & 8.61 & 11.59 & 14.21 & 16.67  \\
    & \mathcal{E}_M (T,\Delta t)  & 6.26 & 8.80 & 11.15 & 13.16 \\ \midrule
    \multirow{2}{*}[-5pt]{Ellipsoid $\frac{1}{2}$} & \mathcal{E}_C (T,\Delta t)  & 9.63 & 11.98 & 14.20 & 15.99 \\
    & \mathcal{E}_M (T,\Delta t)  & 7.87 & 9.87 & 11.87 & 13.86 \\ \midrule
    \multirow{2}{*}[-5pt]{Ellipsoid $\frac{1}{4}$} & \mathcal{E}_C (T,\Delta t)  &  7.17 & 9.77 & 12.25 & 14.58  \\
    & \mathcal{E}_M (T,\Delta t)  & 5.36  & 7.85 & 10.26 & 12.52  \\
     \bottomrule
  \end{tabular}
  \mcaption{tbl:trapz_convdt}{Trapezoidal method self-convergence}{Error terms for center and rotation matrix evolution of three-object swimmers at final time $T = 2\pi$ for the trapezoidal method, for $p=8$ and $\Delta t \in \{ \frac{2\pi}{16},\frac{2\pi}{32},\frac{2\pi}{64},\frac{2\pi}{128}\}$.}
\end{table}

\begin{table}[!htb]
  \newcolumntype{M}{>{\centering\let\newline\\\arraybackslash\hspace{0pt}$}c<{$}}
  \newcolumntype{N}[1]{>{\centering\let\newline\\\arraybackslash\hspace{0pt}}D{.}{.}{#1}}
  \centering
  \small
  \begin{tabular}{N{3.5} N{3.2} N{3.2} N{3.2} N{3.2} N{3.2}}\toprule
    &
    & \multicolumn{1}{c}{\parbox[c]{.6in}{\centering $\Delta t=\frac{2\pi}{16}$}}
    & \multicolumn{1}{c}{\parbox[c]{.6in}{\centering $\Delta t=\frac{2\pi}{32}$}}
    & \multicolumn{1}{c}{\parbox[c]{.6in}{\centering $\Delta t=\frac{2\pi}{64}$}}
    & \multicolumn{1}{c}{\parbox[c]{.6in}{\centering $\Delta t=\frac{2\pi}{128}$}}
    \\ \midrule
    \multirow{2}{*}[-5pt]{Sphere} & \mathcal{E}_C (T,\Delta t)  & 19.75 & 23.77 & 27.78 & 31.80  \\
    & \mathcal{E}_M (T,\Delta t)  & 21.49 & 25.50 & 29.51 & 33.50 \\ \midrule
    \multirow{2}{*}[-5pt]{Ellipsoid $\frac{1}{2}$} & \mathcal{E}_C (T,\Delta t)  & 18.32 & 22.33 & 26.34 & 30.34 \\
    & \mathcal{E}_M (T,\Delta t)  & 18.48 & 22.49 & 26.51 & 30.52 \\ \midrule
    \multirow{2}{*}[-5pt]{Ellipsoid $\frac{1}{4}$} & \mathcal{E}_C (T,\Delta t)  &  17.01 & 21.01 & 25.02 & 29.01  \\
    & \mathcal{E}_M (T,\Delta t)  & 16.94  & 21.00 & 25.03 & 29.06  \\
     \bottomrule
  \end{tabular}
  \mcaption{tbl:rk4_convdt}{Runge-Kutta 4 method self-convergence}{Error terms for center and rotation matrix evolution of three-object swimmers at final time $T = 2\pi$ for the explicit Runge-Kutta method of order $4$, for $p=8$ and $\Delta t \in \{ \frac{2\pi}{16},\frac{2\pi}{32},\frac{2\pi}{64},\frac{2\pi}{128}\}$.}
\end{table}

By comparing the columns in \pr{tbl:trapz_convdt} and \pr{tbl:rk4_convdt}, we confirm that for both center and rotation matrix evolution, the proposed explicit Runge-Kutta time-stepping schemes display second and fourth order convergence, respectively. 

\section{Conclusions}
\label{sec:conclusions}
We presented a new BIE formulation for the mobility problem in three dimensions. Its main advantages are that auxiliary sources inside each rigid body---required by classical approaches such as the completed double-layer formulation \cite{pozrikidis1992boundary}---are obviated, thereby reducing the number of unknowns, and that hypersingular integrals are avoided when computing the hydrodynamic stresses. The extra linear conditions in the case of the scattered field problem i.e., zero net force and torque constraints on the particles, are imposed by simply modifying the BIEs for the densities. Unlike classical formulations, the translational and rotational velocities of each particle are not the fundamental variables but are determined by integrating the surface velocities of each particle.

The numerical method proposed is spectrally-accurate in space via the use of spherical harmonics to represent the geometries and densities and a pole-rotation based singular integration scheme. We couple Stokes FMM with a preconditioned GMRES solver to obtain a solution method that scales linearly with the number of particles. A fourth-order explicit Runge-Kutta method was then used to evolve the particle positions. We presented a series of numerical results that verified the convergence and scaling analysis of our method. 

The mobility problem arises naturally in many physical models, and we demonstrated the applicability of our solution process to sedimentation problems, low-Re swimmer models and magneto-rheological fluid models. The method naturally extends to another important setting, that of when a background flow is imposed. It can be in the form of free-space velocity field i.e., $\vc[u](\vc[x]) \rightarrow \vc[u]_\infty(\vc[x]) \ \ \mathrm{as} \ \ |\vc[x]| \rightarrow \infty$, or disturbance velocity field in the case of flows through constrained geometries. In both cases, due to the linearity of Stokes flow, the imposed field simply is added to the boundary integral representation of the scattered field e.g., $\vc[u]_{sc} = \vc[u]_\infty + \cA[S]_{\Gamma}[\vc[\mu]]$. The BIEs are then modified accordingly. 

Robust simulation of {\em dense} particle suspensions requires several more algorithmic components. When the particles are located close to each other, the interaction terms become nearly singular and specialized quadratures are required to achieve uniform high-order convergence. We plan to extend the recently developed quadrature-by-expansion (QBX) \cite{klockner2013quadrature} to three dimensions. While the recent work of \cite{klinteberg2016fast} implemented a QBX method for three-dimensional problems, the particle shapes were restricted to spheroids. Adaptive time-stepping also becomes important to resolve close interactions. Finally, to enable simulations through periodic geometries, we plan to extend the recently developed periodization scheme in \cite{marple2015fast} to the mobility problem. This scheme uses free-space kernels only and the FMM can be used to accelerate the discrete sums as we do in this work.

\section{Acknowledgements}
EC and SV acknowledge support from NSF under grants DMS-1224656, DMS-1454010
and DMS-1418964 and a Simons Collaboration Grant for Mathematicians \#317933. 
LG and MR acknowledge support from the U.S. Department of Energy Office of Science, Office of Advanced Scientific Computing Research, Applied Mathematics program under Award Number 
DEFGO288ER25053 and the Office of the Assistant Secretary of Defense for Research and 
Engineering and AFOSR under NSSEFF Program Award FA9550-10-1-0180.

\bibliography{rbs3d}
\end{document}